\documentclass[11pt]{amsart}
\usepackage[T1]{fontenc}
\numberwithin{equation}{section}

\usepackage{float}
\usepackage{tikz-cd}
\usepackage{amsmath}
\usepackage{amsfonts}
\usepackage{amsthm}
\usepackage{amssymb}
\usepackage{graphicx}
\usepackage{hyperref}

\usepackage{float}
\usepackage{tikz-cd}
\usepackage{amsmath}
\usepackage{amsfonts}
\usepackage{amsthm}
\usepackage{amssymb}
\usepackage{graphicx}
\usepackage{hyperref}
\usepackage{enumerate}
\usepackage{comment}

\usepackage{geometry}
\usepackage{layout}
\usepackage[utf8]{inputenc}
\usepackage{amsmath}
\usepackage{graphicx}
\usepackage{amsmath}
\usepackage{amssymb}
\usepackage{amsfonts}
\usepackage{graphicx}
\usepackage{url}
\usepackage{MnSymbol}
\usepackage{arydshln}
\usepackage{outlines}
\usepackage{tikz-cd}

\usepackage{subfig}
\usepackage{amsthm}

\newcommand{\Fbar}{\overline{\mathbb{F}}}

\renewcommand{\ss}{\mathrm{ss}}

\newcommand{\ZZ}{\mathbb{Z}}

\newcommand{\cA}{\mathcal{A}}
\newcommand{\cO}{\mathcal{O}}

\newcommand{\cM}{\mathcal{M}}
\newcommand{\cS}{\mathcal{S}}
\newcommand{\cJ}{\mathcal{J}}
\newcommand{\F}{\mathbb{F}}
\newcommand{\PGL}{\mathrm{PGL}}

\renewcommand{\epsilon}{\varepsilon}

\renewcommand{\P}{\mathbb{P}}

\renewcommand{\cJ}{\mathcal{J}}

\newcommand{\dirlim}{\varinjlim}

\newcommand{\GL}{\mathrm{GL}}

\newcommand{\rank}{\mathrm{rank}}

\newcommand{\Aut}{\mathrm{Aut}}

\newcommand{\prank}{p\text{-}\rank}
\newcommand{\tworank}{2\text{-}\rank}

\usepackage{multirow}

\newcommand{\<}{\left \langle}
\renewcommand{\>}{\right \rangle}

\usepackage{xcolor}

\makeatletter
\newtheorem*{rep@thm}{\rep@title}
\newcommand{\newrepthm}[2]{%
\newenvironment{rep#1}[1]{%
 \def\rep@title{#2 \ref{##1}}%
 \begin{rep@thm}}%
 {\end{rep@thm}}}
\makeatother

\newtheorem{thm}{Theorem}[section]

\newtheorem{lem}[thm]{Lemma}
\newtheorem{prop}[thm]{Proposition}
\newtheorem{cor}[thm]{Corollary}

\newrepthm{thm}{Theorem}

\theoremstyle{definition}

\theoremstyle{remark}
\newtheorem{rem}[thm]{Remark}
\newtheorem{exmp}[thm]{Example}

\begin{document}

\author{Du\v san Dragutinovi\'c}
\keywords{Curves, Jacobians, supersingular, automorphisms, finite fields, characteristic $2$}
\address{Mathematical Institute, Utrecht University, 3508 TA Utrecht, The Netherlands}
\email{d.dragutinovic@uu.nl}

\title[An unusual family of supersingular curves of genus five]{An unusual family of supersingular curves \\  of genus five in characteristic two}

\begin{abstract}
We construct a family of smooth supersingular curves of genus $5$ in characteristic $2$ with several notable features: its dimension matches the expected dimension of any component of the supersingular locus in genus $5$, its members are non-hyperelliptic curves with non-trivial automorphism groups, and each curve in the family admits a double cover structure over both an elliptic curve and a genus-$2$ curve. We also provide an explicit parametrization of this family.
\end{abstract}
\maketitle

\section{Introduction}
\label{sec:intro}
An elliptic curve $E$ defined over a field $k$ of positive characteristic $p > 0$ is supersingular if $\#E[p](\overline{k}) = 1$; that is, if the neutral element is the only $p$-torsion point of $E$. 
By \cite{deuring}, there are only finitely many isomorphism classes of supersingular elliptic curves over $k = \Fbar_p$. In particular, up to isomorphism, there is a unique supersingular elliptic curve over $\Fbar_2$, for example, $E: y^2 + y = x^3$.
A smooth curve $C$ of genus $g\geq 2$ over $k$ is supersingular if there exists an isogeny $\cJ_C \to E^g$ defined over~$\overline{k}$, where $\cJ_C$ is the Jacobian of $C$ and $E$ is any supersingular elliptic curve over $\overline{k}$. 
A famous result of van der Geer and van der Vlugt (\cite[Theorem 2.1]{vdgvdv}) shows that for every~$g \geq 2$, there exists a smooth supersingular curve $C_{GV, g}$ of genus $g$ defined over~$k = \Fbar_2$. 

In this article, we study the locus $\cM_g^{\ss}$ of supersingular curves in the moduli space $\cM_g$ of smooth genus-$g$ curves over $k$ in the case $g = 5$ and $k = \Fbar_2$. The cases $g \leq 4$ have already been discussed in the literature. 
For $g = 2$, it follows from \cite[Section 2]{igusa} that $\cM_2^{\ss}$ is irreducible of dimension $1$ over $\Fbar_2$. For $g = 3$, it was shown in \cite[Theorem 1.12, Theorem 5.6]{oort_hess} that $\cM_3^{\ss}$ is pure of dimension $2$ and that there are no smooth hyperelliptic supersingular curves of genus $3$ over $\Fbar_2$.
For $g = 4$, it was shown in \cite[Theorem 3.4]{dra2} that $\cM_4^{\ss}$ is pure of dimension $3$ over $\Fbar_2$.

Before presenting our results, we summarize a few properties relevant for our discussion:
\begin{itemize}
    \item In \cite[Proposition~4.4]{vdgvdv} a positive-dimensional family $\mathcal{F}_{GV, g}$ of supersingular curves $C_{GV, g}$ over~$\Fbar_2$ for any $g > 1$ was constructed. \emph{In particular, $\dim \mathcal{F}_{GV, 5} = 2$.} 

    \item \emph{If the supersingular locus $\cM_5^{\ss}$ in $\cM_5$ is non-empty in characteristic $p$ (which is true for $p = 2$), any component of $\cM_5^{\ss}$ has dimension at least~$3$, and the expected dimension of $\cM_5^{\ss}$ equals $3$.} Namely, by \cite[Theorem 4.9]{lioort}, the supersingular locus $\cS_g \subseteq \cA_g$ of principally polarized abelian varieties of dimension $g$ is pure of dimension $\lfloor g^2/4 \rfloor$ over $\Fbar_p$ for any $g \geq 1$ and $p > 0$. Since $\cA_g$ is a smooth stack, the intersection $\cJ_g \cap \cS_g$ has dimension at least $3$ when $g = 5$, where $j: \cM_g \to \cA_g$ is the Torelli map and $\cJ_g = \overline{j(\cM_g)}$, and it would be exactly $3$ if $\cJ_g$ is dimensionally transverse to $\cS_g \subseteq \cA_g$. 

    \item By a conjecture of Oort (\cite[Problem 4]{emo01}), given $g > 1$ and $p > 0$,  the generic point of every component of the supersingular locus $\cS_g$ of $\cA_g$ over $\Fbar_p$ has automorphism group $\{\pm 1\}$. So far, this conjecture was shown to be true for $g = 2, 3$ if~$p>2$ with counterexamples for $(g, p) \in \{(2, 2), (3, 2)\}$, for $g = 4$ if $p>0$, and $g = 2n$ if~$p>3$ for any $n \in \ZZ_{>2}$; see \cite{ibukiyama}, \cite{kp}, \cite{oort_hess} \cite{kyy}, \cite{dra4}, \cite{vkcfy}. An analog of this conjecture for curves would state the following: \emph{Assuming that the generic point of a component of the supersingular locus of $\cM_g$ corresponds to a non-hyperelliptic curve, its automorphism group is trivial.}
\end{itemize}

Smooth curves of genus $5$ can be either hyperelliptic, trigonal, or non-hyperelliptic and non-trigonal, whose canonical model in $\P^4$ is a complete intersection of three quadrics; see~\cite[Section IV.5]{hag}. 
Using this classification, the data of all smooth curves of genus $5$ over $\F_2$, up to isomorphism, were determined in \cite{dra1}; similar data for curves of genus~$4$ (resp.~$6$) over~$\F_2$ were obtained in \cite{xarles} (resp.~\cite{huangkedlayalau}). According to the data, all non-hyperelliptic supersingular curves of genus $5$ over $\F_2$ with a non-trivial $\F_2$-automorphism group are non-trigonal, and their canonical models lie on the quadric $X^2 + YZ = 0$ in~$\P^4$, up to a linear change of coordinates (in the dataset, this quadric is actually taken to be~$Y^2 + XZ + YZ = 0$).

This intriguing property motivated our study of non-hyperelliptic, non-trigonal curves of genus $5$ over $\Fbar_2$ whose canonical model lies on the quadric $X^2 + YZ = 0$. As our main result, we prove the existence of a family of smooth supersingular curves of genus $5$ over $\Fbar_2$ exhibiting rather unusual properties. More precisely, we have the following theorem.

\begin{thm}
Let $b_1, b_2, b_3 \in \Fbar_2$, $b_3 \neq 0$ and let $C = V(q_1, q_2, q_3) \subseteq \P^4$, where
\begin{equation}
\begin{cases}
q_1 = X^2 + YZ, \\
q_2 = T^2 + ZT + XY,\\
q_3 = U^2 + UZ + b_1ZT + b_2XZ + b_3XT + b_3^2YT, 
\end{cases} 
\label{eqn:supersingular_3dimfamily_genus5}
\end{equation}
Then, $C$ is a smooth canonical curve of genus $5$ over $\Fbar_2$ which is supersingular and satisfies the following properties: 
\begin{enumerate}
    \item $\ZZ/2\ZZ \times \ZZ/2\ZZ \subseteq \Aut(C)$;
    \item There exists a double cover $C \to E$, where $E$ is an elliptic curve;
    \item There exists a double cover $C \to D$, 
    where $D$ is a smooth curve of genus $2$.
    \item The $a$-number of $C$ equals $a(C) = 5 - \mathrm{rank}_{\Fbar_2}(F|_{H^1(C, \cO_C)}) = 2$, where $F$ is the Frobenius operator on $C$;
\end{enumerate}
Furthermore, let $\cS$ denote the family which consists of the isomorphism classes of all curves~$C$ as in \eqref{eqn:supersingular_3dimfamily_genus5}. Then $\dim \cS = 3$.   
\label{thm:supersingular_family_genus5_char2_dim3}
\end{thm}

Note that $\dim \cS$ equals the expected dimension of the supersingular locus $\cM_5^{\ss}$ in $\cM_5$ and $\dim \cS > \dim \mathcal{F}_{GV, 5}$. Furthermore, the generic point of $\cS$ has a non-trivial automorphism group. Therefore, if the closure~$\overline{\cS}$ of~$\cS$ turns out to be a component of $\cM_5^{\ss}$, it would provide a counterexample to the analog of Oort’s conjecture in the case $g = 5$ and $p = 2$. While we were not able to decide this, we comment more on this possibility in Corollary~\ref{cor:ss_genus5_component_or_not}. Also, it is worth noting that $\cS$ is a sublocus of the loci of double covers of both elliptic curves and curves of genus~$2$.
Lastly, note that by \cite[Corollary~10.3]{lioort}, the generic point of every component of $\cS_g \subseteq \cA_g$ has $a$-number~$1$ for any $g > 1$ and $p > 0$, whereas the $a$-number of~$\cS$ equals~$2$, resulting in another unusual property of~$\cS$. We also present further discussion after proving this theorem in Section~\ref{sec:proof}. (While Section \ref{sec:genus5ssauto} explains how we obtain the family~$\cS$, verifying that $\cS$ has the property stated in Theorem \ref{thm:supersingular_family_genus5_char2_dim3} relies only on Lemmas \ref{lem:smooth_dim1} and~\ref{lem:description_of_all_isomphs} and Propositions~\ref{prop:genus5ss_dimension_automorphisms} and \ref{prop:genus5ss_doublecover_bielliptic_bi2}.) 

For some of the computations, we use \cite{sagemath} and refer to \url{https://github.com/DusanDragutinovic/Examples}.
\\
\\
\noindent \textbf{Outline.} In Section \ref{sec:prelim}, we recall definitions and results on invariants of curves in positive characteristic, such as the $p$-rank and Newton polygon, and explain how we compute the $p$-rank and geometric automorphism groups of the curves we study. In Section \ref{sec:genus5ssauto}, we consider the family of all non-hyperelliptic, non-trigonal genus-$5$ curves over $\Fbar_2$ whose canonical models lie on the intersection of three quadrics in $\P^4$, one of which has rank $3$. 
In Lemma~\ref{lem:reduce_num_of_part}, we reduce the number of parameters needed to describe non-isomorphic curves in this family; in Lemma \ref{lem:2rank0_genus5_firstequation}, we determine which have $p$-rank $0$; and in Proposition \ref{prop:genus5ss_dimension_automorphisms}, we compute their geometric automorphism groups. This analysis leads to the construction of the supersingular family $\cS$. 
Finally, in Section \ref{sec:proof}, we prove Theorem \ref{thm:supersingular_family_genus5_char2_dim3} and discuss some relevant questions.
\\
\\
\textbf{Acknowledgment}
The author is grateful to Carel Faber and Valentijn Karemaker for helpful discussions and valuable comments, and to the anonymous referees for their careful reading and constructive remarks that improved this article.  The author is supported by the Mathematical Institute of Utrecht University. 

\section{Preliminaries}
\label{sec:prelim}

In this section, we review definitions and results concerning invariants of curves in positive characteristic, including the $p$-rank, $a$-number, Newton polygon, and the supersingular Newton polygon. We also describe our methods for computing the $p$-rank and geometric automorphism groups of genus-$5$ curves over $\Fbar_2$ that arise as complete intersections of three quadrics in $\P^4$, one of which has rank $3$.

\subsection{Invariants of curves in positive characteristic} Let $K$ be a perfect field of characteristic $p>0$ and let $\sigma: K \to K$ the Frobenius of $K$. In the case when $K$ is algebraically closed, we will denote it by $K = k$. 
Furthermore, let $C$ be a smooth curve over an algebraically closed field $k$ of genus $g \geq 2$ and let $\cJ_C$ be its Jacobian, which is a principally polarized abelian variety of dimension $g$ over $k$.

\subsubsection{The $p$-rank and the $a$-number of a smooth curve}
Let 
\begin{equation}
F: H^{1}(C, \cO_C) \to H^{1}(C, \cO_C) 
\label{eqn:frobenius_operator}
\end{equation}
be the $\sigma$-linear Frobenius operator on $C$. 
The $\prank$ of $C$, denoted by $f(C)$, is the semisimple rank of $F$, i.e., $$f(C) = \rank_{k}(F^{g}\mid_{H^1(C, \cO_C)}), \quad 0 \leq f(C) \leq g, $$ while the $a$-number of $C$, denoted by $a(C)$, is the number $$a(C) = g - \rank_k(F\mid_{H^1(C, \cO_C)}), \quad 0 \leq a(C) \leq g.$$
Equivalently, one can define the $p$-rank of $C$ as the number $0 \leq f = f(C) \leq g$ such that~${\#\cJ_C[p](k) = p^f}$. In this article, we will usually write $$\tworank(C) = f(C)$$ for the $\tworank$ of a curve $C$ defined over $k = \Fbar_2$.

Given a matrix $M\in k^{g\times g}$, for some $g \geq 2$, we denote by $M^{\sigma^n}$ the matrix obtained from $M$ by raising each of its entries to the $p^{n}$-th power, so if $M = (m_{ij})$, then $M^{\sigma^n} = (m_{ij}^{p^n})$, $n\geq 1$. In the case when $C$ is a smooth curve of genus $g$, the matrix of the Frobenius operator \eqref{eqn:frobenius_operator} with respect to a fixed basis of $H^1(C, \cO_C)$ is called the Hasse-Witt matrix, denoted by~$H = H_C$; the Hasse-Witt matrix of a smooth curve $C$ is unique up to relation $H \sim SH(S^{\sigma})^{-1},$ corresponding to the changes of bases of $H^{1}(C, \cO_C)$ induced by transition matrices $S \in \GL_{g}(k)$. If $H$ is the Hasse-Witt matrix of $F$ as in \eqref{eqn:frobenius_operator}, then the operator $F^n = F \circ \ldots \circ F$ is represented by the matrix $$H\cdot H^{\sigma}\cdot \ldots \cdot H^{\sigma^{n-1}}, n\geq 1.$$ In particular, $$f(C) = \rank_k (H\cdot H^{\sigma}\cdot \ldots \cdot H^{\sigma^{g-1}})\quad \text{and}\quad a(C) = g - \rank_k(H).$$

\subsubsection{The Newton polygon of a smooth curve and supersingular curves}
\label{subsub:np_of_smooth_curves}

Let $\cJ_C[p^{\infty}] = \dirlim \cJ_C[p^n]$ be the $p$-divisible group of $\cJ_C$. By the Dieudonn\'e-Manin classification \cite{manin}, there are certain $p$-divisible groups $G_{c, d}$ for $c, d\geq 0$ relatively prime integers and there is an isogeny of $p$-divisible groups \[\cJ_C[p^{\infty}] \sim \bigoplus_{\lambda = \frac{d}{c + d}}G_{c, d},\] 
for a unique choice of \textit{slopes} $\lambda$. If a slope $\lambda$ appears in this decomposition with multiplicity~$m$, then so does the slope $1 - \lambda$. The Newton polygon of $C$ is defined as the collection~$\xi(C)$ of those $\lambda$ counted with multiplicities. We usually think of Newton polygon $\xi(C)$ of $C$ as the 
connected union of finitely many line segments in the plane, which is lower convex and breaks at points with integer coordinates, starting at $(0, 0)$ and ending at $(2g, g)$. The slopes $\lambda$ are the slopes of these line segments.

If $C$ is a smooth curve of genus $g$ defined over a finite field $\F_p$ with $p$ prime, we can compute its Newton polygon $\xi(C)$ as follows. Consider the zeta function of $C$, that is 
$Z(C/\F_p, t) = \exp(\sum_{n \geq 1} \#C(\F_{p^n}) t^n/n)$. By the Weil conjectures, 
$$Z(C/\F_p, t) = \frac{L(C/\F_p, t)}{((1 - t)(1 - pt)},$$ where $L(C/\F_p, t) = \sum_{i=0}^{2g} a_i t^i$ is the $L$-polynomial of $C$. The Newton polygon of $C$ then equals the lower convex hull of the points $(i, \mathrm{val}_p(a_i))$ for $i = 0, \dots, 2g,$ where $\mathrm{val}_p$ denotes the normalized $p$-adic valuation; see, e.g., \cite[Section 2]{dra2}.

If the Newton polygon $\xi(C)$ is a straight line starting at $(0, 0)$ and ending at $(2g, g)$, or altenatively if all slopes of $\xi(C)$ equal $\frac{1}{2}$, we say that $C$ is \emph{supersingular}. Alternatively, it follows from \cite[1.4]{lioort} that a smooth curve~$C$ of genus $g$ over a field $K$ is supersingular if and only if there exists an isogeny $$\cJ_C \overset{\sim}{\to} E^g$$ over~$k = \overline{K}$, where $E$ is any supersingular elliptic curve over $k$.

The number of slopes $0$ in $\xi(C)$ equals the $p$-rank of $C$. Consequently, if $C$ is supersingular, then $f(C) = 0$. The reverse implication holds for $g = 2$, but not necessarily for~$g \geq 3$.

\subsubsection{The purity of de Jong and Oort} 
\label{subsec:dejongoort_purity}
By the purity theorem of de Jong and Oort, if a family $\mathcal{F}$ of smooth curves of genus $g$ with constant Newton polygon $\xi$ contains in its closure in $\cM_g$ elements with a different Newton polygon, then this already occurs in codimension~$1$. That is, there exists a family $\mathcal{F}' \subseteq \cM_g$ of dimension $\dim \mathcal{F}' = \dim \mathcal{F} - 1$, whose generic Newton polygon is $\xi' \neq \xi$, such that $\mathcal{F}' \subsetneq \overline{\mathcal{F}}$. This result is well known and holds much more generally; for the precise statement, see \cite[Theorem 4.1]{dejongoort}. A similar result also holds for the $p$-rank (and can in fact be seen as a consequence of de Jong-Oort purity for suitably chosen Newton polygons). 

\subsection{Computing Hasse-Witt matrices and automorphism groups of curves}

We will be interested in explicitly computing the $p$-rank and the automorphism group of certain genus-$5$ curves in characteristic $p = 2$ throughout this article. To do this, we will use the description provided below.

To compute the Hasse-Witt matrices of the curves studied in this article, we will use the description provided by Kudo-Harashita in \cite{kudoharashita}, as follows. 
Given a perfect field $K$ of characteristic $p>0$ and a smooth curve $C = V(q_1, \ldots, q_{n - 1})$ of genus $g$ which is a complete intersection in $\P^n$, defined by homogeneous polynomials $q_1, \ldots, q_{n - 1} \in K[X_0, \ldots, X_n]$, let us assume that $\sum_{i\neq j}\deg(q_{i})\leq n$ for every $1 \leq j \leq n - 1$.  
By \cite[Proposition B.2.2]{kudoharashita}, if we write $(q_1\ldots q_{n - 1})^{p - 1} = \sum c_{i_0, \ldots, i_n} X_0^{i_0}\ldots X_n^{i_n}$ and 
\begin{small}
$$\{(j_0, \ldots, j_n) \in (\ZZ_{<0}^{n + 1}): \sum_{m = 0}^nj_m = - \sum_{m = 1}^{n - 1}\deg(q_m)\} = \{(j_0^{(1)}, \ldots, j_n^{(1)}), \ldots, (j_0^{(g)}, \ldots, j_n^{(g)})\},$$     
\end{small}
then the Hasse-Witt matrix of $C$ is given by 
$$\begin{pmatrix}
c_{-j_0^{(1)}p + j_0^{(1)}, \ldots, -j_n^{(1)}p + j_n^{(1)}} & \ldots & c_{-j_0^{(g)}p + j_0^{(1)}, \ldots, -j_n^{(g)}p + j_n^{(1)}}\\

\vdots & \ddots & \vdots\\

c_{-j_0^{(1)}p + j_0^{(g)}, \ldots, -j_n^{(1)}p + j_n^{(g)}} & \ldots  & c_{-j_0^{(g)}p + j_0^{(g)}, \ldots, -j_n^{(g)}p + j_n^{(g)}}
 \end{pmatrix}.$$
 
Now, we specialize this description to the case of our interest: $p = 2$ and $g = 5$: a non-hyperelliptic non-trigonal curve $C$ of genus $5$ over $k = \Fbar_2$ can be canonically embedded in $\P^4$ as a complete intersection $C = V(q_1, q_2, q_3)$ where $q_1, q_2, q_3$  are homogeneous polynomials in $k[X, Y, Z, T, U]$ of degrees $$\deg(q_1) = \deg(q_2) = \deg(q_3) = 2.$$ If we write $q_1\cdot q_2\cdot q_3 = \sum c_{i_0, i_1, i_2, i_3, i_4}X^{i_0}Y^{i_1}Z^{i_2}T^{i_3}U^{i_4}$, the Hasse-Witt matrix of $C$ is given by 
\begin{equation}
H = \begin{pmatrix}
c_{2, 1, 1, 1, 1} & c_{0, 3, 1, 1, 1} & c_{0, 1, 3, 1, 1}& c_{0, 1, 1, 3, 1}& c_{0, 1, 1, 1, 3}\\
c_{3, 0, 1, 1, 1} & c_{1, 2, 1, 1, 1} & c_{1, 0, 3, 1, 1}& c_{1, 0, 1, 3, 1}& c_{1, 0, 1, 1, 3}\\ 
c_{3, 1, 0, 1, 1} & c_{1, 3, 0, 1, 1} & c_{1, 1, 2, 1, 1}& c_{1, 1, 0, 3, 1}& c_{1, 1, 0, 1, 3}\\
c_{3, 1, 1, 0, 1} & c_{1, 3, 1, 0, 1} & c_{1, 1, 3, 0, 1}& c_{1, 1, 1, 2, 1}& c_{1, 1, 1, 0, 3}\\
c_{3, 1, 1, 1, 0} & c_{1, 3, 1, 1, 0} & c_{1, 1, 3, 1, 0}& c_{1, 1, 1, 3, 0}& c_{1, 1, 1, 1, 2}
\end{pmatrix}. 
\label{eqn:hasse-witt-cone-genus5-kudoharashita}
\end{equation}

In the rest of this section, we present an observation we will use when computing the automorphism group of curves we are interested in. 

Given any prime power $q$, it was shown in Appendix~B of the 
preprint 
version of the paper \cite{fabergrantham} that 
all matrices in $\mathrm{GL}_n(\F_q)$ that fix the quadratic polynomial $X_1X_2 + \ldots + X_{m - 2}X_{m - 1} + X_m^2 \in \F_q[X_1, \ldots, X_n]$, $m\leq n$, are of the form
$\begin{pmatrix}
A_1 &  0\\
A_2 &  A_3
\end{pmatrix}$, where $A_1$ is an element of $\mathrm{GL}_{m}(\F_q)$ that fixes $X_1X_2 + \ldots + X_{m - 2}X_{m - 1} + X_m^2$ as an element of $\F_q[X_1, \ldots, X_m]$, $A_2$ is any $(n - m) \times m$ matrix with entries in $\F_q$, and $A_3 \in \mathrm{GL}_{n - m}(\F_q)$. 

Consider a homogeneous quadratic polynomial $q_1 = X^2 + YZ$ in $\Fbar_2[X, Y, Z]$.
The set of elements in $\PGL_3(\Fbar_2)$ which fix $q_1$ is of the form $$ \begin{pmatrix}
ad - bc & ac & bd \\ 
0 & a^2 & b^2\\ 
0 & c^2 & d^2
\end{pmatrix},$$ for any $a, b, c, d \in \Fbar_2$ such that $ad - bc \neq 0$. It follows all elements of $\PGL_5(\Fbar_2)$ that fix $X^2 + YZ$ as an element of the polynomial ring $\Fbar_2[X, Y, Z, T, U]$ are of the form 
\begin{equation}
\begin{pmatrix}
ad - bc & ac & {bd} & 0& 0\\ 
0 & a^2 & b^2 & 0& 0\\ 
0 & c^2 & d^2 & 0& 0\\ 
e_1 & f_1 & g_1 & h_1 & i_1\\
e_2 & f_2 & g_2 & h_2 & i_2

\end{pmatrix},
\label{eqn:elts_pgl5_fixing_cone}
\end{equation}
for some $a, b, c, d, e_j, f_j, g_j, h_j, i_j \in \Fbar_2$ for $j \in \{1, 2\}$ such that $h_1i_2 - i_1h_2 \neq 0$ and $ad - bc \neq 0$. (Namely, if $M \in \PGL_5(\Fbar_2)$ is a matrix that fixes $X^2 + YZ \in \F_2[X, Y, Z, T, U]$ as an element of the polynomial ring $\Fbar_2[X, Y, Z, T, U]$, then $M \in \PGL_5(\F_q)$ for some $\F_2 \subseteq \F_q \subseteq \Fbar_2$ and we use the preceding discussion to see that $M$ will be as in \eqref{eqn:elts_pgl5_fixing_cone}.)

\section{Curves of genus $5$ in characteristic $2$ lying on $X^2 + YZ = 0$ in $\P^4$}
\label{sec:genus5ssauto}

By \cite[Section IV.5]{hag}, every smooth curve of genus $5$ is either hyperelliptic, trigonal, or non-hyperelliptic and non-trigonal, in which case its canonical model in $\P^4$ is a complete intersection of three quadrics $$\{q_1 = 0, q_2 = 0, q_3 = 0\} \subseteq \P^4.$$ Using this description, \cite{dra1} generates data for all smooth genus $5$ curves over $\F_2$, classified up to $\F_2$-isomorphism. For each curve, the number of $\F_{2^n}$-points for $1 \leq n \leq 5$ (and hence its Newton polygon; see Subsection \ref{subsub:np_of_smooth_curves}) as well as its $\F_2$-automorphism group were determined.
  
Let us focus on smooth supersingular curves of genus $5$ over $\F_2$. Among the supersingular curves over $\F_2$, all hyperelliptic ones ($8$ in total) have an $\F_2$-automorphism group of size $2$, generated by the hyperelliptic involution, while all trigonal ones ($14$ in total) have trivial $\F_2$-automorphism groups (see \cite[Tables 6 and 7]{dra1}).
Among the non-hyperelliptic, non-trigonal supersingular curves, $14$ have trivial $\F_2$-automorphism groups (see \cite[Table~8]{dra1}), while the remaining $14$ have non-trivial groups (see \cite[Table 9]{dra1}).  

Remarkably, all $14$ non-hyperelliptic, non-trigonal supersingular curves $C$ with non-trivial $\F_2$-automorphism admit a model $\{q_1 = 0, q_2 = 0, q_3 = 0\} \subseteq \P^4$, where $q_1$ can be chosen as $q_1 = Y^2 + XZ + YZ$. After the change of coordinates ($X \mapsto X + Y, Y \mapsto X, Z \mapsto Z$), one may take $q_1 = X^2 + YZ$. Moreover, we observe that in many cases (in fact, in all of them, although we will not prove this) the curve admits a model in which $q_2$ involves four variables $X, Y, Z, T$ and does not involve the variable $U$. Below, we present three examples illustrating this observation.

\begin{exmp}
Let $C_1$ be the curve given by
\begin{small}
$C_1: \begin{cases} Y^2 + XZ + YZ = 0 \\ T^2 + U^2 + TX + TY + UY + Y^2 + TZ + XZ = 0 \\ T^2 + TX + UX + TY + Y^2 + TZ + UZ + Z^2 = 0 \end{cases}$\end{small} in $\P^4$, and let $C_2$ be the curve given by 
\begin{small}
$C_2: \begin{cases} Y^2 + XZ + YZ = 0 \\ T^2 + U^2 + UX + X^2 + Y^2 + TZ + UZ + XZ + Z^2 = 0 \\ U^2 + X^2 + XY + Y^2 + UZ + YZ = 0 \end{cases}$\end{small} in $\P^4$, and let $C_3$ be the curve given by
\begin{small}
$C_3: \begin{cases} Y^2 + XZ + YZ= 0 \\ U^2 + TX + UX + X^2 + XY + UZ + XZ + Z^2= 0 \\ T^2 + TX + UX + X^2 + Y^2 + TZ + XZ + YZ + Z^2= 0 \end{cases}$ 
\end{small} in $\P^4$. Curves $C_i = \{q_1 = 0, q_2 = 0, q_3 = 0\}$, for $i = 1, 2, 3$, appear on the list \cite[Table 9]{dra1} of non-hyperelliptic, non-trigonal supersingular curves $C$ of genus $5$ over $\F_2$ with non-trivial $\F_2$-automorphism; their $\F_2$-automorphism groups are of the cardinality $\#\mathrm{Aut}_{\F_2}(C_i) = 2^i$.

As already mentioned, they all lie on the quadric $q_1 = Y^2 + XZ + YZ = 0$. Let us show that, for each $i$, there is a model $$\{q_1' = 0, q_2' = 0, q_3' = 0\}$$ of $C_i = \{q_1 = 0, q_2 = 0, q_3 = 0\}$,  where $q_1' = q_1 = Y^2 + XZ + YZ$ and $q_2'$ is a polynomial in only $4$ variables (moreover, we can take them to be $X, Y, Z, T$). Note that $C_1$ lies on the quadric $q_2' = q_2 + q_3 =  U^2 + UX + UY + UZ + XZ + Z^2 = 0$, while $C_2$ lies on the quadric $q_2' =  q_2 = U^2 + X^2 + XY + Y^2 + UZ + YZ = 0$. (The change of coordinates $(T \mapsto U, U \mapsto T)$ establishes $q_2' \in k[X, Y, Z, T]$ for $i = 1, 2$.) Lastly, note that $C_3$ lies on $q_2' = q_2 + q_3 = U^2 +  T^2 +  UZ + TZ + Y^2 + XY + YZ = 0$. The change of coordinates $(T \mapsto T + U, U \mapsto U)$ fixes $q_1' = q_1$ and establishes a model of $C_3$ which lies on $q_1' = 0$ and $q_2' = T^2 + TZ + Y^2 + XY + YZ = 0$. 
\end{exmp}

Motivated by the preceding discussion, let us consider a non-hyperelliptic non-trigonal curve $C$ of genus $5$ defined over $k = \Fbar_2$ whose canonical model is defined by
\begin{equation}
C: \{q_1 = 0, q_2 = 0, q_3 = 0\} \text{ in } \P^4_{(X:Y:Z:T:U)},    
\label{eqn:starting_genus5_family_on_a_cone}
\end{equation}
where $q_1 = X^2 + YZ$ and \begin{align*}
q_2 = a_1XY + a_2XZ + a_3XT + a_4Y^2 +  a_5YZ + a_6YT + a_7Z^2 +  a_8ZT + T^2,
\end{align*}
\begin{align*}
q_3 &= b_1XY + b_2XZ + b_3XT + b_4Y^2 + b_5YZ + b_6YT + b_7Z^2 + b_8ZT \\
    &+ b_9XU + b_{10}YU + b_{11}ZU + b_{12}TU + U^2.     
\end{align*}
for some $a_i, b_j \in k$, $1\leq i\leq 8$, $1\leq j \leq 12$.

Let $C$ be a variety as in \eqref{eqn:starting_genus5_family_on_a_cone}, and assume that it is a smooth curve of genus $5$. Then $C$ is canonically embedded in $\P^4$, and any isomorphism $\phi: C \cong C'$, where $C' \subseteq \P^4$ is another smooth canonical curve of genus $5$, is induced by a projective transformation $M \in \PGL_5(k)$: $$\phi = \phi_M: (X, Y, Z, T, U)^t \mapsto M\cdot (X, Y, Z, T, U)^t.$$ Given~$M \in \PGL_5(k)$ and a homogeneous quadratic polynomial $q' \in k[X, Y, Z, T, U]$, we denote by $M.q'$ the result of the action of $M$ on $q'$, induced by $\phi_M$. In particular, note that $M.q_1 = q_1$ for $q_1$ as in \eqref{eqn:starting_genus5_family_on_a_cone} and $M\in \PGL_5(k)$ as in \eqref{eqn:elts_pgl5_fixing_cone}.

In this section, we study the family of curves $C$ given in \eqref{eqn:starting_genus5_family_on_a_cone}, with parameters $a_i$ and~$b_j$, in order to arrive at the family $\cS$ from Theorem \ref{thm:supersingular_family_genus5_char2_dim3}. The rough strategy is as follows. In Subsection \ref{sub:part1}, we reduce the number of parameters needed to capture all isomorphism classes (Lemma \ref{lem:reduce_num_of_part}) and identify those with $\tworank = 0$ (Lemma \ref{lem:2rank0_genus5_firstequation}). In Subsection \ref{sub:part2}, we describe properties of the resulting curves (Lemmas \ref{lem:smooth_dim1} and \ref{lem:description_of_all_isomphs}), and in particular study their automorphisms (Propositions \ref{prop:genus5ss_dimension_automorphisms} and \ref{prop:genus5ss_doublecover_bielliptic_bi2}).

\subsection{Reducing the number of parameters}
\label{sub:part1}
In this subsection, we will focus on finding a model of $C$ as in \eqref{eqn:starting_genus5_family_on_a_cone}, i.e., an isomorphism $\phi: C \cong C'$, 
where $C'$ is a smooth canonical curve of genus $5$ in $\P^4$
that also lies on the quadric $X^2 + YZ = 0$. For that purpose, we will focus on isomorphisms $\phi = \phi_M$ induced by matrices $M$ as in \eqref{eqn:elts_pgl5_fixing_cone}, which fix the quadric~$X^2 + YZ$. 
Note that $X^2 + YZ \in k[X, Y, Z, T, U]$ is, up to a change of coordinates, the unique quadric of rank $3$; see Appendix B of the preprint version of \cite{fabergrantham}.

In the following two claims, we will assume that the variety $C \subseteq \P^4$ as in \eqref{eqn:starting_genus5_family_on_a_cone} is actually a smooth canonical curve of genus~$5$ to arrive at a simpler model of this curve that we will use afterwards. We present our first lemma below.

\begin{lem}
\label{lem:reduce_num_of_part}
Let $C$ be a variety defined by \eqref{eqn:starting_genus5_family_on_a_cone} for some $a_i, b_j \in k$, $1\leq i\leq 8$, $1\leq j \leq 12$. Assume that $C$ is a smooth canonical curve of genus $5$ over $k$. 
Then, after a possible change of coordinates, we may assume that $C$ has the following model:
\begin{equation}
C: \{q_1 = 0, q_2 = 0, q_3 = 0\} \text{ in } \P^4,    
\label{eqn:semimodel_genus5_ff2}
\end{equation}
where $q_1 = X^2 + YZ$, $q_2 = XY + a_3XT + ZT + T^2$, and $$q_3 = b_2XZ + b_3XT + b_6YT + b_8ZT + b_9XU + b_{10}YU + b_{11}ZU + b_{12}TU + U^2.$$
\end{lem}
\begin{proof}
We start with $C$ as in \eqref{eqn:starting_genus5_family_on_a_cone} and consider the changes of coordinates induced by $M$ as in \eqref{eqn:elts_pgl5_fixing_cone} to obtain a simpler model for $C$.\footnote{For certain computations used below, we refer to \url{https://github.com/DusanDragutinovic/Examples}.} 
Recall that $M q_1 = q_1$, and choose the entries of $M$ as follows: $a = d = h_1 = i_2 = 1$ and $b = c = i_1 = h_2 = 0$. Then, modulo $q_1$, $M.q_2 = q_2'$ and $M.q_3 = q_3'$, where $q_2'$, and $q_3'$ are as in \eqref{eqn:starting_genus5_family_on_a_cone}, defined by some coefficients that we denote by $a_{j}'$ and $b_{j}'$, respectively. Moreover, modulo $q_1$, we have: $$a_4' = f_1^2 + f_1a_6 + a_4, \quad a_7' = g_1^2 + g_1a_8 + a_7, \quad a_5' = e_1^2 + e_1a_3 + g_1a_6 + f_1a_8 + a_5, $$
$$b_4' = f_2^2 + f_2(f_1b_{12} + b_{10}) + (f_1b_6 + b_4), \quad b_7' = g_2^2 + g_2(g_1b_{12} + b_{11}) +  (g_1b_8 + b_7), \text{ and}$$
$$b_5' = e_2^2 + e_2(e_1b_{12} + b_9) + f_2(g_1b_{12} + b_{11}) + g_2(f_1b_{12} + b_{10})  + e_1b_3 + g_1b_6 + f_1b_8 + b_5.$$
By appropriately choosing $e_1, f_1, g_1, e_2, f_2, g_2$ in $M$, we can make all the considered coefficients $a_i', b_i', i \in \{4, 5, 7\}$ vanish while preserving the shapes of the quadrics $q_2$ and $q_3$. Therefore, we can assume that $a_4 = a_5 = a_7 = b_4 = b_5 = b_7 = 0$, i.e., that $C$ has a model
\begin{equation}
C: \{q_1 = 0, q_2 = 0, q_3 = 0\} \text{ in } \P^4,    
\label{eqn:1}
\end{equation}
where $q_1 = X^2 + YZ$, $q_2 = a_1XY + a_2XZ + a_3XT + a_6YT + a_8ZT + T^2$ and $$q_3 = b_1XY + b_2XZ + b_3XT + b_6YT + b_8ZT + b_9XU + b_{10}YU + b_{11}ZU + b_{12}TU + U^2.$$

Consider the change of coordinates of $C$ as in \eqref{eqn:1} induced by a matrix $M$ as in \eqref{eqn:elts_pgl5_fixing_cone}. We can always choose $e_1, f_1, g_1$ (resp. $e_2, f_2, g_2$) in $M$ appropriately, regardless of the values of $a, b, c, d$ (resp. $a, b, c, d, e_1, f_1, g_1$), so that the change of coordinates of $C$ is again of the form \eqref{eqn:1}. Moreover, we may assume $a_6 = 0$: if $a_6 = 0$, nothing needs to be done; otherwise, we choose $a$ such that $a_6a^2 + a_3ac + a_8c^2 = 0$. This leads us to the model: \begin{equation}
C: \{q_1 = 0, q_2 = 0, q_3 = 0\} \text{ in } \P^4,    
\label{eqn:2}
\end{equation}
where $q_1 = X^2 + YZ$, $q_2 = a_1XY + a_2XZ + a_3XT + a_8ZT + T^2$ and $$q_3 = b_1XY + b_2XZ + b_3XT + b_6YT + b_8ZT + b_9XU + b_{10}YU + b_{11}ZU + b_{12}TU + U^2.$$

Now, we consider the model \eqref{eqn:2} of $C$. To avoid the singularity at $(0:1:0:0:0) \in C$, we must have $a_1 \neq 0$. Consider the change of coordinates of $C$ induced by $M$ as in \eqref{eqn:elts_pgl5_fixing_cone} with $i_1 = h_2 = c = f_1 = 0$ and $h_1 = i_2 = 1$. Then, modulo $q_1$, we get: 
\begin{align*}
 M.q_2 &= a^3da_1XY + (b^2 ada_1 + b de_1a_3 + ad^3a_2 + adg_1a_3 + d^2e_1a_8)XZ + ada_3XT \\ &+ (bda_3 + d^2a_8)ZT + T^2 + (a^2bda_1 + ade_1a_3 + e_1^2)YZ \\
  &+  (bd(b^2a_1 + d^2a_2) + g_1(bda_3 + d^2a_8) +g_1^2)Z^2.  
\end{align*}
Denote $a_2'= b^2 ada_1 + b de_1a_3 + ad^3a_2 + adg_1a_3 + d^2e_1a_8$, $a_5' = a^2bda_1 + ade_1a_3 + e_1^2,$ and $a_7' = bd(b^2a_1 + d^2a_2) + g_1(bda_3 + d^2a_8) +g_1^2$. Note that $a \neq 0$, $d \neq 0$, and $a_1 \neq 0$. 
We claim that, by appropriately choosing $b, e_1, g_1$, we can make $a_2' = a_5' = a_7' = 0$. 
Indeed: 
\begin{itemize}
    \item If $a_3 = 0$, first choose $e_1$ such that $e_1^2 = a^2bda_1$ to get $a_5' = 0$. Note that the term with $g_1$ in $a_2'$ vanishes since $a_3 = 0$ and solve $a_2'^2 = 0$ (where we replace $e_1^2$ by $a^2bda_1$) for $b$. Lastly, solve $a_7' = 0$ for $g_1$.
    \item If $a_3 \neq 0$, first choose $b = \frac{e_1^2 + ade_1a_3}{a^2da_1}$ to get $a_5' = 0$. Then, (after writing $b$ in terms of~$e_1$)    
    choose $g_1 = \frac{b^2 ada_1 + b de_1a_3 + ad^3a_2 + d^2e_1a_8}{ada_3}$ to get $a_2' = 0$. Lastly, solve $a_7' = 0$ for~$e_1$. 
\end{itemize} 
Finally, we choose $e_2, f_2, g_2$ (as before, this is possible regardless of the values of $a, b, c, d, e_1$, $f_1, g_1$) so that $q_3' = M.q_3$ modulo $q_1$ has the same form as $q_3$ in \eqref{eqn:2}. After these choices, we obtain the following model of $C$: 
\begin{equation}
C: \{q_1 = 0, q_2 = 0, q_3 = 0\} \text{ in } \P^4,    
\label{eqn:3}
\end{equation}
where $q_1 = X^2 + YZ$, $q_2 = a_1XY + a_3XT + a_8ZT + T^2, a_1 \neq 0$, and $$q_3 = b_1XY + b_2XZ + b_3XT + b_6YT + b_8ZT + b_9XU + b_{10}YU + b_{11}ZU + b_{12}TU + U^2.$$

This shows that any curve $C$ in \eqref{eqn:starting_genus5_family_on_a_cone} has a model of the form \eqref{eqn:3}; that is, we may assume $a_2 = a_4 = a_5 = a_6 = a_7 = b_4 = b_5 = b_7 = 0$ and $a_1 \neq 0$ in the starting model. Moreover, to avoid a singularity at $(0:0:1:0:0)$, we may assume $a_8 \neq 0$. Lastly, by taking $M$ as in~\eqref{eqn:elts_pgl5_fixing_cone} with $b = c = e_1 = f_1 = g_1 = i_1 = e_2 = f_2 = g_2 = h_2 = 0$,  $h_1 = i_2 = 1$, and $a, d$ such that 
$$a^3d = \frac{1}{a_1} \quad \text{ and } d^2 = \frac{1}{a_8},$$
we can assume that $a_1 = a_8 = 1$ to get $C$ as in \eqref{eqn:3} with $q_2 = XY + a_3XT + ZT + T^2$.

Finally, consider the transformation $M$ as in \eqref{eqn:elts_pgl5_fixing_cone} defined by $$a = d = h_1 = i_2 = 1, \quad b = c = e_1 = f_1 = g_1 = e_2 = f_2 = g_2 = 0,$$ and note that $\{q_1'= 0, q_2' = 0, q_3' = 0\}$ and $\{q_1'= 0, q_2' = 0, q_3' - b_1q_2' = 0\}$ define the same curve in $\P^4$, where $M.q_i = q_i'$ for $i = 1, 2, 3$. Therefore, the choice of $h_2$ such that $h_2^2 + b_{12}h_2 + b_1 = 0$, leads us to the model of $C$ as stated in \eqref{eqn:semimodel_genus5_ff2}.
\end{proof}

Since we are interested in studying supersingular curves lying on the quadric $X^2 + YZ = 0$ in $\P^4$, we first determine which curves $C$ defined by \eqref{eqn:semimodel_genus5_ff2} have $\tworank = 0$. We have the following lemma.

\begin{lem}
\label{lem:2rank0_genus5_firstequation}
Let $C$ be a variety defined by \eqref{eqn:semimodel_genus5_ff2}. Assume that $C$ is a smooth canonical curve of genus $5$. If $\tworank(C) = 0$ then $C$ has a model of the form $\{q_1 = 0, q_2 = 0, q_3 = 0\}$, where $q_1 = X^2 + YZ$, $q_2 = XY + ZT + T^2$, and $q_3$ equals 
\begin{equation*}
q_3 = b_2XZ + b_3XT + b_6YT + b_8ZT + b_{11}ZU + U^2,  
\end{equation*}
for some $b_6 \neq 0$.
Moreover, if $b_{11} \neq 0$, we may assume that 
\begin{equation}
q_3 = b_2XZ + b_3XT + b_6YT + b_8ZT + ZU + U^2.
\label{eqn:genus5_2rank0_4dim_family}
\end{equation}
\end{lem}
\begin{proof}
We compute\footnote{For certain computations used below, we refer to \url{https://github.com/DusanDragutinovic/Examples}.}  the Hasse-Witt matrix $H = H_C$ of $C$ as in \eqref{eqn:semimodel_genus5_ff2}  using \eqref{eqn:hasse-witt-cone-genus5-kudoharashita} as follows: 
$$H = \begin{pmatrix}
a_3b_9 + b_{10}& 0& b_{11} & b_{12} & 0\\
a_3b_{11} + b_9 & a_3b_{10} + b_{12}&  0 & 0 &  0\\
a_3b_{10} + b_{12} &  0 & a_3b_{11} + b_9 & 0 & 0\\
b_{11}  & b_{10}  &  0 & a_3b_{12} + b_9 & 0\\
b_8 &  b_6 & b_2 & b_3 & a_3
\end{pmatrix}.$$
The condition $\tworank(C) = 0$ is equivalent to $HH^{\sigma}H^{\sigma^2}H^{\sigma^3}H^{\sigma^4} = 0$. We will show that this is equivalent to $a_3 = b_9 = b_{10} = b_{12} = 0$. 
The $(5, 5)$-entry of the matrix $HH^{\sigma}H^{\sigma^2}H^{\sigma^3}H^{\sigma^4}$ equals $a_3^{31}$, and therefore $\tworank(C) = 0$ implies $a_3 = 0$. Furthermore, we claim that $b_{12} = 0$. If $b_{12} \neq 0$, we can scale $U$ using the transformation $U \mapsto b_{12}U$ and then divide out $q_3$ by~$b_{12}^2$; this is equivalent to assuming $b_{12} = 1$. Then, if we denote by $F_1$ and $F_2$ the $(2,2)$- and $(3,2)$-entries of $HH^{\sigma}H^{\sigma^2}H^{\sigma^3}H^{\sigma^4}$, respectively, it follows that $F_1 + b_9F_2 = 1$. This implies that $F_1$ and $F_2$ cannot both be equal to $0$, which means that $\tworank(C) \neq 0$. Therefore, $\tworank(C) = 0$ forces $b_{12} = 0$. Finally, assuming $a_3 = b_{12} = 0$, we find that the $(1,1)$- and $(3,3)$-entries of $HH^{\sigma}H^{\sigma^2}H^{\sigma^3}H^{\sigma^4}$ are $b_{10}^{31}$ and $b_9^{31}$, respectively, so that $\tworank(C) = 0$ implies $b_9 = b_{10} = 0$. Conversely, $a_3 = b_9 = b_{10} = b_{12} = 0$ implies that $HH^{\sigma}H^{\sigma^2}H^{\sigma^3}H^{\sigma^4} = 0$.

Let $C = V(q_1, q_2, q_3) \subseteq \P^4$ be as in \eqref{eqn:semimodel_genus5_ff2} with $a_3 = b_9 = b_{10} = b_{12} = 0$. Consider the Jacobian matrix of $C$ 
at $Q = (x_0:x_1:x_2:x_3:x_4) \in C$,  with $x_0 = X, x_1 = Y, x_2 = Z, x_3 = T, x_4 = U$: 
$$(\partial q_i/ \partial x_j)(Q) =  \begin{pmatrix}
0 & x_2 & x_1 & 0 & 0 \\
x_1 & x_0 & x_3 & x_2 & 0 \\
b_2x_2 + b_3x_3 & b_6x_3 & b_2x_0 + b_8x_3 + b_{11}x_4 & b_3x_0 + b_6x_1 + b_8x_2 & b_{11}x_2 \\
\end{pmatrix}.$$
To avoid a singularity at $Q = (0:1:0:0:0)$, we must have $b_6 \neq 0$. Furthermore, if $b_{11} \neq 0$, then $C$ is nonsingular at $Q$ with $x_2 \neq 0$. In this case, after scaling $U \mapsto b_{11}U$ and dividing by $b_{11}^2$, we may assume $b_{11} = 1$, leading to the model in \eqref{eqn:genus5_2rank0_4dim_family}.
\end{proof}

\subsection{Properties of curves $\{q_1 = 0, q_2 = 0, q_3 = 0\} \subseteq \P^4$ with $q_3$ as in \eqref{eqn:genus5_2rank0_4dim_family}}
\label{sub:part2}
We focus on curves defined by $\{q_1 = 0, q_2 = 0, q_3 = 0\}$ in $\P^4$, where $q_1 = X^2 + YZ, q_2 = XY +ZT + T^2$, and $q_3$ is as in \eqref{eqn:genus5_2rank0_4dim_family}. First, in the lemma below, we give a criterion for when $C$ is actually a smooth canonical curve of genus $5$.

\begin{lem} Let $C = V(q_1, q_2, q_3) \subseteq \P^4$, where $q_1 = X^2 + YZ, q_2 = XY +ZT + T^2$, and $q_3$ is as in~\eqref{eqn:genus5_2rank0_4dim_family}. If $b_6 \neq 0$, then $C$ is a smooth canonical curve of genus $5$ over $k = \Fbar_2$. 
\label{lem:smooth_dim1}
\end{lem}
\begin{proof}
Let $C = V(q_1, q_2, q_3) \subseteq \P^4$, with $q_1 = X^2 + YZ$, $q_2 = XY + ZT + T^2$, and $q_3$ as in~\eqref{eqn:genus5_2rank0_4dim_family}. Moreover, assume that $b_6 \neq 0$.
First, we show that $C$ is an irreducible algebraic variety, compute its function field $\kappa(C)$, and establish that $\mathrm{trdeg}_k(\kappa(C)) = 1$, which implies that $C$ is an irreducible algebraic curve in $\P^4$. Second, we show that $C$ is smooth, which completes the proof of the claim.

Since $C = V(q_1, q_2, q_3) \subseteq V(q_1, q_2) \subseteq V(q_1) \subseteq \P^4$, every component of $C$ has dimension~$\geq 1$. To show that $C$ is irreducible, it suffices to find an open dense subset $U_C = U \cap C \subseteq C$ and show that $U_C$ is irreducible. Take $U = \{Z \neq 0\} \cap \{b_3X + b_6Y + b_8Z \neq 0\}$ and note that $(\P^4 - U)\cap C$ consists of finitely many points, so that $U_C = U \cap C \subseteq C$ is an open dense subset. Writing $x = X/Z, y = Y/Z, t = T/Z$ and $u = U/Z$ and noting that $y = x^2$ and $t = \frac{u^2 + u + b_2x}{b_6x^2 + b_3x + b_8Z}$ on $U_C$, we see that $U_C$ is isomorphic to the curve $f(x, u) = 0$ in $\mathbb{A}^2_{x, u}$ with $$f(x, u) = u^4 + (1 + h(x))u^2 + h(x)u + (b_2^2x^2 + b_2xh(x) + x^3h(x)^2),$$ where $h(x) = b_6x^2 + b_3x + b_8$. Assume that $f(x, u) = (u^3 + \alpha_2u^2 + \alpha_1u + \alpha_0)(u + \beta_0)$ or $f(x, u) = (u^2 + \gamma_1u + \gamma_0)(u^2 + \delta_1u + \delta_0)$ for some polynomials $\alpha_i, \beta_i, \gamma_i,$ and $\delta_i$ in $k[x]$. By considering these as equalities of polynomials in $u$ with coefficients in $k[x]$ and comparing the coefficients of $u^i$ for $i \geq 1$, we find  $\alpha_2 = \beta_0, \alpha_1 = \beta_0^2 + 1 + h$, and $\alpha_0 = \beta_0^3 + \beta_0 + \beta_0h + h$, as well as $\gamma_1 = \delta_1, \gamma_0 = 1 + h + \delta_1^2 + \delta_0$, and $(\delta_1 + 1)(\delta_1^2 + \delta_1 + h) = 0$ (which implies that $\delta_1 = 1$ or $\deg_x(\delta_1) = 1$), respectively.  
Finally, comparing the coefficients of $u^0$ in the above equalities and analyzing their degrees (as elements of $k[x]$), we use the fact that $\deg_x(b_2^2x^2 + b_2xh(x) + x^3h(x)^2) = 7$ (since $b_6 \neq 0$) to conclude that there are no $\beta_0, \delta_0 \in k[x]$ satisfying them. This shows that $C$ is irreducible and, moreover, that $\kappa(C) \cong k(x,u)/(f(x,u))$ with $\mathrm{trdeg}_k(\kappa(C)) = 1$. Hence~$C$ is an irreducible algebraic curve in $\P^4$.

Consider the Jacobian matrix of $C = V(q_1, q_2, q_3) \subseteq \P^4$ 
at $Q = (x_0:x_1:x_2:x_3:x_4) \in C$,  with $x_0 = X, x_1 = Y, x_2 = Z, x_3 = T, x_4 = U$: 
$$(\partial q_i/ \partial x_j)(Q) =  \begin{pmatrix}
0 & x_2 & x_1 & 0 & 0 \\
x_1 & x_0 & x_3 & x_2 & 0 \\
b_2x_2 + b_3x_3 & b_6x_3 & b_2x_0 + b_8x_3 + x_4 & b_3x_0 + b_6x_1 + b_8x_2 & x_2 \\
\end{pmatrix}.$$
If $x_2 \neq 0$, then $\rank_k (\partial q_i/ \partial x_j)(Q) = 3$, so $C$ is smooth at $Q$. Otherwise, the conditions $x_2 = 0$ and $Q \in C$ imply $Q = (0:1:0:0:0)$. In this case, $\rank_k (\partial q_i/ \partial x_j)(Q) = 3$ if and only if~$b_6 \neq 0$, which holds by assumption.

To summarize: $C$ is a nonsingular irreducible curve that is a complete intersection of three quadrics in $\P^4$. Therefore, by \cite[Example IV.5.5.3]{hag}, $C$ is a smooth canonical curve of genus $5$ over $k = \Fbar_2$.
\end{proof}

In the following lemma, we show that any isomorphism between the curves considered in the preceding lemma is induced by a matrix $M \in \PGL_5(k)$ as in \eqref{eqn:elts_pgl5_fixing_cone}.

\begin{lem}
Let $C = V(q_1, q_2, q_3) \subseteq \P^4$ and $C' = V(q_1, q_2, q_3') \subseteq \P^4$ be two varieties, where $q_1 = X^2 + YZ, q_2 = XY +ZT + T^2$, $q_3 = b_2XZ + b_3XT + b_6YT + b_8ZT + ZU + U^2$, and $q_3' = b_2'XZ + b_3'XT + b_6'YT + b_8'ZT + ZU + U^2$ for some $b_2, b_2', b_3, b_3', b_6, b_6', b_8, b_8' \in k = \Fbar_2$ such that $b_6, b_6' \neq 0$. Then, any isomorphism $C \cong C'$ is induced by a projective transformation $M \in \PGL_5(k)$ that fixes $q_1$, i.e., by a matrix $M$ as in \eqref{eqn:elts_pgl5_fixing_cone}.  
\label{lem:description_of_all_isomphs}
\end{lem}
\begin{proof}
Note that both $q_3$ and $q_3'$ are of the form \eqref{eqn:genus5_2rank0_4dim_family}, so that $b_6, b_6' \neq 0$ and Lemma \ref{lem:smooth_dim1} imply that both $C \subseteq \P^4$ and $C' \subseteq \P^4$ are smooth canonical curves of genus $5$ over $k = \Fbar_2$. Hence any isomorphism $\phi: C \to C'$ is induced by some $M \in \PGL_5(k)$, i.e. $\phi = \phi_M$. We claim that $M$ is of the form \eqref{eqn:elts_pgl5_fixing_cone}, that is, $M.q_1 = q_1$.

Let $N = \begin{pmatrix}
\alpha_1 & \alpha_2 & \alpha_3 & \alpha_4 & \alpha_5 \\
\beta_1 & \beta_2 & \beta_3 & \beta_4 & \beta_5 \\
\gamma_1 & \gamma_2 & \gamma_3 & \gamma_4 & \gamma_5 \\
\end{pmatrix} \in k^{3 \times 5}$ be the matrix consisting of the top three rows of $M$, and note that $\phi_M: (X, Y, Z)^t \mapsto N \cdot (X, Y, Z, T, U)^t$.
Since $\phi_M: C \cong C'$, there exist $\alpha, \beta, \gamma \in k$ such that $M.q_1 = \alpha q_1 + \beta q_2 + \gamma q_3'$. It suffices to show that $\beta = \gamma = 0$; then, after rescaling (since $M \in \PGL_5(k)$), we may take $\alpha = 1$. By comparing the coefficients of $XU$, $YU$, $ZU$, and $TU$ in the equality $M.q_1 = \alpha q_1 + \beta q_2 + \gamma q_3'$, we obtain the following identities:
\begin{equation}
\beta_1\gamma_5 = \beta_5\gamma_1, \quad \beta_2\gamma_5 = \beta_5\gamma_2, \quad \beta_3\gamma_5 + \beta_5\gamma_3 = \gamma, \quad \beta_4\gamma_5 = \beta_5\gamma_4,
\label{eqn:XUYUTUZU}
\end{equation}
and similarly, by comparing the coefficients of $XY, XT, YT$, and $ZT$ we get: 
\begin{equation}
\beta_1\gamma_2 + \beta_2\gamma_1 = \beta, \quad \beta_1\gamma_4 + \beta_4\gamma_1 = b_3'\gamma, \quad \beta_2\gamma_4 + \beta_4\gamma_2 = b_6'\gamma, \quad \beta_3\gamma_4 + \beta_4\gamma_3 = \beta + b_8'\gamma.
\label{eqn:XYXTYTZT}
\end{equation}
There are three possibilities, and we show that $\beta = \gamma = 0$ in each of them: 
\begin{itemize}
    \item If $\beta_5 = \gamma_5 = 0$, then \eqref{eqn:XUYUTUZU} implies $\gamma = 0$. If moreover $\beta_4 = \gamma_4 = 0$, then the fourth equality in \eqref{eqn:XYXTYTZT} implies $\beta = 0$. Otherwise, $\beta_4 \neq 0$ or $\gamma_4 \neq 0$. If $\beta_4 \neq 0$ (resp.~$\gamma_4 \neq 0$), then the second and the third quality in \eqref{eqn:XYXTYTZT} imply $\gamma_1 = \frac{\gamma_4}{\beta_4}\beta_1$ and $\gamma_2 = \frac{\gamma_4}{\beta_4}\beta_2$ (resp.~$\beta_1 = \frac{\beta_4}{\gamma_4}\gamma_1$ and $\beta_2 = \frac{\beta_4}{\gamma_4}\gamma_2$), while the first one consequently gives us $\beta = 0$; 
    \item If $\beta_5 = 0$ and $\gamma_5 \neq 0$, then \eqref{eqn:XUYUTUZU} implies $\beta_1 = \beta_2 = \beta_4 = 0$, while the first and the third equality in \eqref{eqn:XYXTYTZT} and $b_6' \neq 0$ result in $\beta = \gamma = 0$. 
    \item Lastly, if $\beta_5 \neq 0$, then \eqref{eqn:XUYUTUZU} implies $\gamma_1 = \frac{\gamma_5}{\beta_5}\beta_1$, $\gamma_2 = \frac{\gamma_5}{\beta_5}\beta_2$ and $\gamma_4 = \frac{\gamma_5}{\beta_5}\beta_4$. Now, the first and the third equality in \eqref{eqn:XYXTYTZT} and $b_6' \neq 0$ imply  $\beta = \gamma = 0$.
\end{itemize}
Therefore, for any isomorphism $\phi = \phi_M: C \cong C'$, induced by $M \in \PGL_5(k)$, we have $M.q_1 = q_1$. It follows that $M$ is of the form \eqref{eqn:elts_pgl5_fixing_cone}.
\end{proof}

Now, we determine when two of these curves are isomorphic to each other. This analysis also leads to a better understanding of their automorphism groups.

\begin{prop}
Let $C$ be a smooth curve of genus $5$ over $k$ whose canonical model in $\P^4$ is given by $\{q_1 = 0, q_2 = 0, q_3 = 0\}$, where $q_1 = X^2 + YZ, q_2 = XY +ZT + T^2$, and $q_3$ is as in~\eqref{eqn:genus5_2rank0_4dim_family}, for some $b_2, b_3, b_6, b_8 \in k$, $b_6 \neq 0$. Then the following statements hold: 
\begin{enumerate}
    \item There are only finitely many smooth curves $C'$ of genus $5$ over $k$ given by the model $\{q_1 = 0, q_2 = 0, q_3' = 0\}$ in $\P^4$ such that $C \cong C'$, where $q_1$ and $q_2$ are as above, and $q_3' = b_2'XZ + b_3'XT + b_6'YT + b_8'ZT + ZU + U^2$, for some $b_2', b_3', b_6', b_8' \in k$, $b_6' \neq 0$.
    \item If $b_6 \neq b_3^2$, then $\Aut(C)$ is a group of order $2$, generated by the automorphism 
    \begin{equation}
     \iota: U \mapsto U + Z.
     \label{eqn:genus5_iota}
    \end{equation}
    If $b_6 = b_3^2$, then $\<\iota\>\times \<\eta\> \cong \ZZ/2\ZZ \times \ZZ/2\ZZ$ is a subgroup of $\Aut(C)$, where $\iota$ is as in~\eqref{eqn:genus5_iota}, and
    \begin{equation}
     \eta: T\mapsto Z + T,\quad  U \mapsto  b_3X + g_2Z + U,
     \label{eqn:genus5_eta}
    \end{equation}
for some $g_2\in k$ such that $g_2^2 + g_2 + b_8 = 0$. Moreover, $\Aut(C) \neq \<\iota\> \times \<\eta\>$ if and only if $b_8^2 + b(b_3^4 + b_3) = 0$ for some $b\in k$ such that $b^3 = 1$.  
\end{enumerate}
\label{prop:genus5ss_dimension_automorphisms}
\end{prop}
\begin{proof}
By Lemma \ref{lem:smooth_dim1} and Lemma \ref{lem:description_of_all_isomphs}, $C \subseteq \P^4$ and $C'\subseteq \P^4$ are smooth canonical curves of genus $5$ and any isomorphism between them is induced by a projective automorphism $M \in \PGL_5(k)$, given by some $$M = M(a, b, c, d, e_1, f_1, g_1, h_1, i_1, e_2, f_2, g_2, h_2, i_2) \text{ as in \eqref{eqn:elts_pgl5_fixing_cone}},$$ with $ad - bc \neq 0$ and $h_1i_2 - h_2i_1 \neq 0$. 

Assume that the isomorphism $C\cong C'$ is given by such an $M$.\footnote{For certain computations used below, we refer to \url{https://github.com/DusanDragutinovic/Examples}.}  Denote by $M.q_2$ the result of the action of $M$ on $q_2$, induced by $(X, Y, Z, T, U)^t \mapsto M\cdot (X, Y, Z, T, U)^t.$ A priori, $(*): M.q_2 = \alpha\cdot q_1 + \beta\cdot q_2 + \gamma\cdot q_3'$ for some $\alpha, \beta, \gamma \in k$. We claim that $\gamma = 0$. Indeed, by comparing the coefficients of $U^2, T^2, YT$, and $YU$ in $(*)$, we find that  $\gamma = i_1^2, \beta = h_1^2$, and $c^2h_1 + i_1^2b_6 = c^2i_1 = 0$, and therefore $\gamma = i_1 = 0$, taking into account that $b_6 \neq 0$. Now, the identity $M.q_2 = \alpha \cdot q_1 + \beta \cdot q_2$ for some $\alpha, \beta \in k$ implies $\beta = h_1^2$ by comparing the coefficients of $T^2$. Since $M \in \PGL_5(k)$ and $\beta \neq 0$, we may assume $\beta = h_1^2 = 1$. In addition to the previously obtained values $\gamma = i_1 = 0$ and $\beta = h_1 = 1$, by considering the coefficients of $YT$, $Y^2$, $ZT$, $XZ$, $XY$, $YZ$ (modulo $q_1$), and $Z^2$, we find $c = f_1 = 0$, $d =  1$, $e_1 = ab^2$, and 
\begin{equation}
a^3 = 1, \quad b(b^3 + 1) = 0, \quad g_1^2 + g_1 + b^3 = 0.   
\label{eqn:genus5_q2}
\end{equation}

Finally, the preceding discussion implies that $(\#): M.q_3 = \alpha'\cdot q_1 + \beta'\cdot q_2 + \gamma' \cdot q_3'$ for some $\alpha', \beta', \gamma' \in k$ with $\gamma' \neq 0$. By comparing the coefficients of $ZU$ and $U^2$ in $(\#)$, we find $\gamma' = i_2^2 = i_2$, so that $i_2 = 1$ as $\gamma' \neq 0$.  Furthermore, by comparing the coefficients of $T^2, Y^2$, $YZ$ (modulo $q_1$), $Z^2, XY, XZ, XT, YT$, and $ZT$ in $(\#)$, we obtain $\beta' = h_2^2$, $f_2 = 0$, and the following identities:
\begin{equation}
e_2^2 = a^2b^2b_3 + a^2g_1b_6
 \label{eqn:genus5_YZ}
\end{equation}
\begin{equation}
g_2^2 + g_2 = b^2g_1b_6 + bg_1b_3 + bb_2 + g_1b_8 
 \label{eqn:genus5_Z^2}
\end{equation}
\begin{equation}
h_2^2 = a^3b^2b_6
 \label{eqn:genus5_XY}
\end{equation}
\begin{equation}
 b_2' = ab^4b_6 + ab^3b_3 + ab^2b_8 + ag_1b_3 + ab_2 + e_2
 \label{eqn:genus5_XZ}
\end{equation}
\begin{equation}
 b_3'= ab_3
 \label{eqn:genus5_XT}
\end{equation}
\begin{equation}
 b_6' = a^2b_6
 \label{eqn:genus5_YT}
\end{equation}
\begin{equation}
b_8' = b^2b_6 + h_2^2 + bb_3 + h_2 + b_8
 \label{eqn:genus5_ZT}
\end{equation}
Given $b_2, b_3, b_6, b_8 \in k, b_6 \neq 0$, we conclude there are only finitely many $$(a, b, c, d, e_1, f_1, g_1, h_1, i_1, e_2, f_2, g_2, h_2, i_2, b_2', b_3', b_6', b_8') \in k^{18}$$ such that the preceding identities are satisfied. Namely, given $a, b, g_1$ such that \eqref{eqn:genus5_q2} is satisfied, one can solve \eqref{eqn:genus5_YZ} for $e_2$, \eqref{eqn:genus5_Z^2} for $g_2$, 
\eqref{eqn:genus5_XY} for $h_2$, and then \eqref{eqn:genus5_XZ} for $b_2'$, \eqref{eqn:genus5_XT} for $b_3'$, \eqref{eqn:genus5_YT} for $b_6'$, and \eqref{eqn:genus5_ZT} for $b_8'$, where all these equations have at most finitely many solutions. 

To compute $\Aut(C)$, we choose $C = C'$ and use the identities (including \eqref{eqn:genus5_q2}-\eqref{eqn:genus5_ZT}) obtained in the first part of the proof of this proposition. In addition to $c = i_1 = f_1 = f_2 = 0$, $d = h_1 = i_2 = 1$, $e_1 = ab^2$, \eqref{eqn:genus5_YT} gives us that $a = 1$ as $b_6 \neq 0$, while substituting $h_2^2 = b^2b_6$ from~\eqref{eqn:genus5_XY} into \eqref{eqn:genus5_ZT} results in $h_2 = b\cdot b_3$. Moreover, we obtain $$b^2(b_3^2 + b_6) = 0,$$
by substituting $h_2 = b\cdot b_3$ in \eqref{eqn:genus5_ZT}.

Note that the choice $b = c = e_1 = e_2 = f_1 = f_2 = g_1 = i_1 = 0$ and $a = d = h_1 = i_2 = g_2 = 1$ satisfies the preceding equalities, and thus it defines an automorphism $\iota$ of $C$, given by 
$$\iota: U \mapsto U + Z;$$
this is an automorphism of order $2$.

If $b_6 \neq b_3^2$, then $b = 0$ and the equalities \eqref{eqn:genus5_q2}, \eqref{eqn:genus5_YZ}, and \eqref{eqn:genus5_XZ} imply $g_1 = 0$. Together with \eqref{eqn:genus5_YZ} and \eqref{eqn:genus5_Z^2}, this further gives $e_2 = 0$ and $g_2(g_2 + 1) = 0$. Choosing $g_2 = 0$ yields a trivial automorphism of $C$, while $g_2 = 1$ gives the automorphism $\iota$ described above. These are all possibilities; hence, in the case $b_6 \neq b_3^2$, $C$ has no additional automorphisms.

Now let $b_6 = b_3^2$. By \eqref{eqn:genus5_q2}, we have $b(b^3 + 1) = 0$ and $g_1^2 + g_1 + b^3 = 0$. 
If $b = g_1 = 0$, then \eqref{eqn:genus5_YZ} and \eqref{eqn:genus5_Z^2} imply $e_2 = 0$ and $g_2(g_2 + 1) = 0$, giving either the trivial automorphism of $C$ when $g_2 = 0$ or the automorphism $\iota$ as described above when $g_2 = 1$. If $b = 0$ and $g_1 = 1$, then \eqref{eqn:genus5_YZ} and \eqref{eqn:genus5_Z^2} imply $e_2 = b_3$ and $g_2^2 + g_2 + b_8 = 0$. Therefore, for any $g_2 \in k$ satisfying $g_2^2 + g_2 + b_8 = 0$, all the equalities considered above hold, so that $$\eta: T\mapsto Z + T, U \mapsto b_3X + g_2Z + U$$ defines an automorphism of $C$ of order $2$. (In fact, once a solution $g_2$ to this equation is fixed, the automorphism corresponding to the other solution $g_2'$ equals $\iota \circ \eta$.)
If $b \neq 0$, then $b^3 = 1$. Now, \eqref{eqn:genus5_q2}, \eqref{eqn:genus5_YZ}, and \eqref{eqn:genus5_XZ} result in $b_8^2 + b (b_3^4 + b_3) = 0$. In particular, if $b_8^2 + b(b_3^4 + b_3) \neq 0$, then $C$ has no additional automorphisms. If $b_8^2 + b(b_3^4 + b_3) = 0$, then the additional automorphisms of $C$ are determined by $$b^3 = 1,\text{ } g_1^2 + g_1 = b^3, \text{ } e_2 = b b_3^2 + b_3 + b^2 b_8 + g_1b_3, \text{ }  \text{ and } g_2^2 + g_2 + (g_1(b^2b_3^2 + bb_3 +  b_8)+ bb_2) = 0.$$  We conclude that $$\Aut(C) = \<\iota\>\times \<\eta\>$$ for $C$ as above with $b_6 = b_3^2$ if and only if $b_8^2 + b(b_3^4 + b_3) \neq 0$.  
\end{proof}

Now, we further investigate the properties of curves $C$ as in the preceding proposition when $b_6 = b_3^2$. In this case, we have found two automorphisms of $C$, denoted by $\iota$ and~$\eta$. As we will see, understanding the quotient maps $$C \to C/\< \iota \>\quad \text{ and } \quad C \to C/\< \eta \>$$ will provide valuable information that we use to better understand $C$ itself. These maps are of degree~$2$ and are determined by the inclusions of function fields $\kappa(C/\< \iota \>) = \kappa(C)^{\< \iota \>} \subseteq \kappa(C)$ (resp.~$\kappa(C/\< \eta \>) = \kappa(C)^{\< \eta \>} \subseteq \kappa(C)$) of $C/\< \iota \>$ (resp.~$C/\< \eta \>$) into the function field of $C$.

\begin{prop}
\label{prop:genus5ss_doublecover_bielliptic_bi2}
Let $C$ be a smooth curve of genus $5$ over $k$ whose canonical model in~$\P^4$ is given by $\{q_1 = 0, q_2 = 0, q_3 = 0\}$, where $q_1 = X^2 + YZ, q_2 = XY +ZT + T^2$, and $q_3 = b_2XZ + b_3XT + b_3^2YT + b_8ZT + ZU + U^2$, for some $b_2, b_3, b_8 \in k$, $b_3 \neq 0$. Let $\iota$ and $\eta$ be elements of $\Aut(C)$ as in \eqref{eqn:genus5_iota} and \eqref{eqn:genus5_eta}, respectively. Then:
\begin{enumerate}
    \item $C/\<\iota\>$ is an elliptic curve, and the map induced by $\eta$ does not act trivially on the function field of $C/\<\iota\>$;
    \item $C/\<\eta\>$ and $C/\<\iota \circ \eta\>$ are two curves of genus $2$.  
\end{enumerate}
\end{prop}
\begin{proof}
We first compute the function field of $C/\<\iota\>$, that is, the subfield $\kappa(C)^{\<\iota\>}$ of $\kappa(C)$ consisting of elements fixed by $\iota$. Write $x = X/Z, y = T/Z, u = U/Z$, and note that $Y/Z = x^2$ because of $q_1 = 0$. The function field of $C$ equals $$\kappa(C) = k(x, y, u)/(y^2 + y + x^3, u^2 + u + b_3^2x^2y + b_3xy + b_2x + b_8y). $$
Since $u'= u^2 + u$ is  fixed by $\iota$, we find that  
\begin{align*}
\kappa(C/\<\iota\>) &= k(x, y, u')/(y^2 + y + x^3, u' + b_3^2x^2y + b_3xy + b_2x + b_8y)\\
&\cong k(x, y)/(y^2 + y + x^3),
\end{align*}
which is the function field of the unique (up to isomorphism) supersingular elliptic curve $E: y^2 + y = x^3$ in characteristic $2$. Moreover, $\eta: y \mapsto y + 1$ and $y$ is not fixed by~$\eta$, so that $\eta$ does not act trivially on $\kappa(C/\<\iota\>)$.

Now, we compute the function field of $C/\<\eta\>$. Note that computing $\kappa(C/\<\iota \circ \eta\>)$ is completely analogous since $\iota\circ \eta: T \mapsto T + Z, U \mapsto U + b_3Z + g_2'Z$, where $g_2'\neq g_2$ is another solution to $g_2^2 + g_2 + b_8 = 0$. Denote $x_1 = X/Z, x_1^2 = Y/Z, t = T/Z$, and $y_0 = U/Z$, using $q_1 = 0$. From $q_3 = 0$, we can express $t$ in terms of $x_1$ and $y_0$, and put it into $q_2 = 0$. In these coordinates, we see that $\eta: y_0 \mapsto y_0 + b_3x_1 + g_2$, where $g_2^2 + g_2 = b_8$, so that $y_1 = y_0(y_0 + b_3x_1 + g_2)$ is an element fixed by $\eta$. It follows that $\kappa(C/\<\eta\>) = k(x_1, y_1)/(y_1^2 + (b_3x_1 + g_2 + 1)y_1 + F_1)$, 
where $$F_1 = b_3^4x_1^7 + b_3^2x_1^5 + (g_2^4 + b_2b_3^2 + g_2^2)x_1^3 + (b_2^2 + b_2b_3)x_1^2 + (g_2^2b_2 + g_2b_2)x_1,$$ or equivalently, $\kappa(C/\<\eta\>) = k(x_2, y_2)/(y_2^2 + y_2 + F_2)$, where $x_2 = b_3x_1 + g_2 + 1$, $y_2 = y_1/ x_2$, and 
\begin{align*}
F_2 &= \frac{1}{b_3^3}x_2^5 + \frac{g_2 + 1}{b_3^3}x_2^4 + \frac{g_2^2}{b_3^3}x_2^3 + \frac{g_2^3 + g_2^2}{b_3^3}x_2^2 + \frac{b_2b_3^2 + g_2^2 + 1}{b_3^3}x_2 + \\
 & + \frac{g_2b_2b_3^2 + g_2^3 + b_2^2b_3 + g_2^2 + g_2 + 1}{b_3^3} +  \frac{g_2b_2 + b_2}{b_3}x_2^{-1} + \frac{g_2^2b_2^2 + b_2^2}{b_3^2}x_2^{-2}.
\end{align*}
Let $c_{-1}$ and $c_{-2}$ denote the coefficients of $x_{2}^{-1}$ and $x_2^{-2}$ in $F_2$, respectively. Clearly, $c_{-1}^2 = c_{-2}$, and therefore,  the function field of $C/\<\eta\>$ equals $$\kappa(C/\<\eta\>) = k(x, y)/(y^2 + y + F_3(x)),$$ where $x = x_2$, $y = y_2 + c_{-1}x^{-1}$, and $F_3(x) \in k[x]$ is a polynomial of degree $5$. We conclude that $C/\<\eta\>$ is a smooth curve of genus $2$. 
\end{proof}

\section{Proof of the main theorem and some remarks}
\label{sec:proof}

We now prove Theorem \ref{thm:supersingular_family_genus5_char2_dim3}, the main result of this article, and provide some further discussion. 
The formulation presented in Section \ref{sec:intro} follows after relabeling $b_1 = b_8$.

\begin{thm}
\label{thm:supersingular_family_genus5_char2_dim3_2}
Let $b_2, b_3, b_8 \in k = \Fbar_2$, $b_3 \neq 0$ and let $C = V(q_1, q_2, q_3) \subseteq \P^4$, where 
\begin{equation}
\begin{cases}
q_1 = X^2 + YZ, \\
q_2 = XY + ZT + T^2,\\
q_3 = b_2XZ + b_3XT + b_3^2YT + b_8ZT + ZU + U^2. 
\end{cases} 
\label{eqn:supersingular_3dimfamily_genus5_2}
\end{equation}
Then, $C$ is a smooth canonical curve of genus $5$ over $\Fbar_2$ which is supersingular and satisfies the following properties:
\begin{enumerate}
    \item $\ZZ/2\ZZ \times \ZZ/2\ZZ \subseteq \Aut(C)$;
    \item There exists a double cover $C \to E$, where $E$ is an elliptic curve;
    \item There exists a double cover $C \to D$, 
    where $D$ is a smooth curve of genus $2$;
        \item The $a$-number of $C$ equals $a(C) = 2$.
\end{enumerate}
Furthermore, let $\cS$ denote the family which consists of the isomorphism classes of all curves~$C$ as in \eqref{eqn:supersingular_3dimfamily_genus5_2}. Then $\dim \cS = 3$.   
\end{thm}

\begin{proof}
Lemma \ref{lem:smooth_dim1} shows that $C$ is a smooth canonical curve of genus $5$, while Lemma~\ref{lem:2rank0_genus5_firstequation} implies that $\tworank(C) = 0$. By the second part of Proposition~\ref{prop:genus5ss_dimension_automorphisms}, it follows that $\<\iota\>\times \<\eta\> \subseteq \Aut(C)$, where $\iota: U \mapsto U + Z$ and $\eta: T \mapsto T + Z$, $U \mapsto U + b_3X + g_2Z$ are two automorphisms of order $2$, where $g_2 \in k$ is such that $g_2^2 + g_2 = b_8$. By Proposition~\ref{prop:genus5ss_doublecover_bielliptic_bi2}, $C \to C/\<\iota\>$ is a double cover of an elliptic curve, while  $C \to~C/\<\eta\>$ and~$C \to C/\<\iota \circ \eta\>$ are two double covers of genus $2$ curves. 

Using the properties above, we can now prove that $C$ is a supersingular curve. Consider the quotient maps $$C \to C/\<\iota\>,\quad C \to C/\<\eta\>,\quad \text{and}\quad C \to C/\<\iota\circ \eta\>.$$ The pullbacks of $\cJ_{C/\<\iota\>}$, $\cJ_{C/\<\eta\>}$, and $\cJ_{C/\<\iota\circ \eta\>}$ to $\cJ_C$ are subvarieties of $\cJ_C$. The induced automorphism $\eta^*$ acts trivially on the pullback of $\cJ_{C/\<\eta\>}$ but not on the pullbacks of $\cJ_{C/\<\iota\circ \eta\>}$ (otherwise, $C/\<\iota\>$ would have genus $2$) or $\cJ_{C/\<\iota\>}$ (by Proposition \ref{prop:genus5ss_doublecover_bielliptic_bi2}).
Consequently, the intersection of the pullback of $\cJ_{C/\<\eta\>}$ to $\cJ_C$ with the pullbacks of $\cJ_{C/\<\iota\>}$ and $\cJ_{C/\<\iota\circ \eta\>}$ is $0$-dimensional. Similarly, the intersection of the pullbacks of $\cJ_{C/\<\iota\>}$ and $\cJ_{C/\<\iota\circ \eta\>}$ to $\cJ_C$ is $0$-dimensional.
Hence, there is an isogeny 
\begin{equation}
\cJ_C \sim \cJ_{C/\<\iota\>}\oplus \cJ_{C/\<\eta\>}\oplus \cJ_{C/\<\iota\circ \eta\>}.
\label{eqn:isogeny_genus5_char2}
\end{equation}
Finally, all the quotient curves are supersingular since they are of genus $g \in \{1, 2\}$ and have $\tworank = 0$, taking into account that $\tworank(C) = 0$. Hence, $C$ is supersingular.

Let us prove the remaining claims. If we denote by $H$ the Hasse-Witt matrix of $C$ as in~\eqref{eqn:hasse-witt-cone-genus5-kudoharashita}, we compute $H =   \begin{pmatrix}
0& 0& 1 & 0 & 0\\
0 & 0&  0 & 0 &  0\\
0 &  0 & 0 & 0 & 0\\
1  & 0  &  0 & 0 & 0\\
0 &  b_3^2 & b_2 & b_3 & 0
\end{pmatrix}$ and conclude that $a(C) = 5 - \rank_k(H) = 2$ as $b_3 \neq 0$. Lastly, consider the family $\cS$ consisting of all isomorphism classes of curves $C$ as in \eqref{eqn:supersingular_3dimfamily_genus5_2}, for $b_2, b_8 \in k$, $b_3 \in k^*$. As a consequence of the first part of Proposition \ref{prop:genus5ss_dimension_automorphisms}, it follows that $\dim \cS = 3$. 
\end{proof}

\begin{rem}
In fact, we can avoid using the condition ${\tworank(C) = 0}$ in the proof of Theorem~\ref{thm:supersingular_family_genus5_char2_dim3}. From the proof of Proposition \ref{prop:genus5ss_doublecover_bielliptic_bi2}, we observe that $C/\langle \iota \rangle$ is the elliptic curve defined by the equation $y^2 + y = x^3$, which is the unique supersingular elliptic curve in characteristic~$2$, up to isomorphism. Curves $C/\langle \eta \rangle$ and $C/\langle \iota \circ \eta \rangle$ are both hyperelliptic curves of genus~$2$ given by equations of the form $y^2 + y = F(x)$, where $F(x) \in k[x]$ is a polynomial of degree~$5$. Moreover, they are both supersingular since \cite[Proposition 4.1]{scholtenzhu} implies that
$$
\tworank(C/\langle \eta \rangle) = \tworank(C/\langle \iota \circ \eta \rangle) = 0.
$$ 
Finally, the isogeny \eqref{eqn:isogeny_genus5_char2} shows that $C$ itself is supersingular.
\end{rem}

\begin{rem}
Another way to conclude that $a(C) = 2$ for curves $C$ as in Theorem \ref{thm:supersingular_family_genus5_char2_dim3} is as follows. By \cite[Lemma 1.2]{looi}, any such curve $C$ admits an even theta characteristic $L$, i.e., a line bundle on $C$ satisfying $L^2 = K_C$, where $K_C$ is the canonical bundle, and such that $\dim H^0(C, L) = 2$. Since $\tworank(C) = 0$, it follows from \cite[Propositions 3.1 and~3.3]{stohrvoloch} that $a(C) = \dim H^0(C, L) = 2.$
\end{rem}

We conclude with a corollary that shows the generic Newton polygon of the locus of all curves as in Proposition \ref{prop:genus5ss_dimension_automorphisms} is not supersingular and discusses the possibility that $\cS$ is a component of $\cM_5^{\ss}$ in $\cM_5$.

\begin{cor}
\label{cor:ss_genus5_component_or_not}
Let $\mathcal{R} \subseteq \cM_5$ be the family of smooth curves of genus~$5$ over $k$ whose canonical model in $\P^4$ is given by $\{q_1 = 0, q_2 = 0, q_3 = 0\}$, where $q_1 = X^2 + YZ$, ${q_2 = XY + ZT + T^2}$, and $q_3$ as in \eqref{eqn:genus5_2rank0_4dim_family}, for some $b_2, b_3, b_8 \in k$ and $b_6 \in k^*$, and let $\cS$ be the sublocus of $\mathcal{R}$ defined by $b_6 = b_3^2$ (i.e., $\cS$ is as in Theorem \ref{thm:supersingular_family_genus5_char2_dim3_2}). Then the following properties hold:
\begin{enumerate}
    \item The generic point of the $4$-dimensional locus $\mathcal{R}$ corresponds to a curve whose Newton polygon is not supersingular. 
    \item At least one of the following two holds:
    \begin{itemize}
        \item There is a component $\Gamma \subseteq \cM_5^{\ss}$ in characteristic $2$ with $\dim \Gamma \geq 4$; or
        \item $\cS \subseteq \cM_5$ is an intersection of $\mathcal{R}$ and a $4$-dimensional locus whose generic Newton polygon has slopes $\frac{2}{5}$ and $\frac{3}{5}$; in this case, $\cS$ is a component of the supersingular locus of $\cM_5$ with a non-trivial generic automorphism group.  
    \end{itemize}
\end{enumerate}
\end{cor}
\begin{proof}
 As we discussed in the first part of Proposition \ref{prop:genus5ss_dimension_automorphisms}, $\mathcal{R}$ is irreducible and~${\dim \mathcal{R} =4}$. By Theorem \ref{thm:supersingular_family_genus5_char2_dim3_2}, $\cS \subseteq \mathcal{R}$ is a divisor whose generic Newton polygon is supersingular. The locus in $\mathcal{R}$ where the Newton polygon is supersingular is closed (see, e.g., \cite[Section~1]{oort_np}), so its closure contains no points with a different Newton polygon. Thus, to show that the generic Newton polygon of $\mathcal{R}$ is not supersingular, it suffices to find a single $[C] \in \mathcal{R}$ with non-supersingular Newton polygon. 
Take $C$ as in \eqref{eqn:genus5_2rank0_4dim_family} with $b_2 = b_3 = b_8 = 0$ and $b_6 = 1$. Then $C$ is a smooth genus $5$ curve over $\F_2$ with $[C] \in \mathcal{R}$. Using \cite{sagemath}, we compute $(\#C(\F_{2^n}))_{1 \leq n \leq 5} = (5, 9, 11, 17, 25)$, which implies that the slopes of the Newton polygon of~$C$ are $\frac{1}{3}, \frac{1}{2}, \frac{1}{2}, \frac{2}{3}$; see, e.g., the description in Subsection \ref{subsub:np_of_smooth_curves}. Hence, the generic Newton polygon of $\mathcal{R}$ is not supersingular.

By Proposition \ref{prop:genus5ss_doublecover_bielliptic_bi2}, every $[C] \in \mathcal{R}$ admits a double cover $C \to E$, where $E$ is an elliptic curve. Since the Newton polygon is invariant under isogeny and $\mathcal{R}$ is contained inside the $2$-rank $0$ locus, only three types can occur on~$\mathcal{R}$: the supersingular one, the polygon with slopes $\frac{1}{3}, \frac{1}{2}, \frac{1}{2}, \frac{2}{3}$, or the polygon with slopes $\frac{1}{4}, \frac{1}{2}, \frac{3}{4}$. 
Thus, the generic Newton polygon of $\mathcal{R}$ must be one of the latter two, and let us assume it has slopes $\frac{1}{3}, \frac{1}{2}, \frac{1}{2}, \frac{2}{3}$ (resp.~$\frac{1}{4}, \frac{1}{2}, \frac{3}{4}$). 
By the purity of de Jong-Oort (see Subsection~\ref{subsec:dejongoort_purity}), \cite[Theorem~4.1]{oort_np} and the smoothness of $\cA_5$, every component of the Newton polygon locus in $\cM_5$ with generic slopes $\frac{1}{3}, \frac{1}{2}, \frac{1}{2}, \frac{2}{3}$ (resp.~$\frac{1}{4}, \frac{1}{2}, \frac{3}{4}$) must have dimension at least $5$ (resp.~$6$), so that $\mathcal{R}$ cannot be such a component. 
Therefore, there is a component $\mathcal{R}'$ of the mentioned Newton polygon locus such that $\dim \mathcal{R}' \geq 5$ (resp.~$\dim \mathcal{R}' \geq 6$) and $\mathcal{R} \subsetneq \mathcal{R}'$. 
Since $\cS$ has codimension at least $2$ (resp.~$3$) in $\mathcal{R}'$, a final application of the purity of de Jong-Oort yields a $4$-dimensional irreducible locus $\cS' \subseteq \cM_5$ (resp.~and a $5$-dimensional $\cS'' \subseteq \cM_5$) such that $\cS \subsetneq \cS' \subsetneq \mathcal{R}'$ (resp. $\cS \subsetneq \cS' \subsetneq \mathcal{S}'' \subsetneq \mathcal{R}'$) and the generic Newton polygons of $\cS'$ (resp.~$\cS''$) and $\mathcal{R}'$ are different. By the classification of Newton polygons in dimension $5$ (see, e.g., \cite[Section 1]{oort_np}), 
the generic Newton polygon of $\cS'$ is either supersingular or has slopes~$\frac{2}{5}$ and $\frac{3}{5}$. 
\end{proof}

\end{document}